\documentclass[a4paper]{article}
\usepackage{amssymb}
\usepackage{amsmath}
\usepackage{mathrsfs}
\usepackage{amsfonts}
\usepackage{color}
\usepackage{vmargin}
\usepackage{amsthm}

\long\def\symbolfootnote[#1]#2{\begingroup%
\def\thefootnote{\fnsymbol{footnote}}\footnote[#1]{#2}\endgroup}

\setmarginsrb{20mm}{20mm}{20mm}{20mm}{10mm}{10mm}{10mm}{10mm}

\def\XXint#1#2#3{{\setbox0=\hbox{$#1{#2#3}{\int}$ }
\vcenter{\hbox{$#2#3$ }}\kern-.6\wd0}}

\long\def\symbolfootnote[#1]#2{\begingroup%
\def\thefootnote{\fnsymbol{footnote}}\footnote[#1]{#2}\endgroup}

\setmarginsrb{20mm}{20mm}{20mm}{20mm}{10mm}{10mm}{10mm}{10mm}

\numberwithin{komcounter}{section}

\theoremstyle{plain}
\newtheorem{theorem}{Theorem}[section]
\newtheorem{prop}[theorem]{Proposition}
\newtheorem{lemma}{Lemma}[section]
\newtheorem{corol}{Corollary}[theorem]
\theoremstyle{definition}

\newtheorem*{remark}{\textnormal{\textbf{Remark}}}

\theoremstyle{remark}

\theoremstyle{definition}

\title{\bf Symmetry and compact embeddings for critical exponents in metric-measure spaces}
\author{ Micha{\l} Gaczkowski$^1$, Przemys{\l}aw G\'{o}rka$^1$\footnote{Email: p.gorka@mini.pw.edu.pl} and Daniel J. Pons$^2$\footnote{Email: dpons@unab.cl ; pons.dan@gmail.com } \\
\small{$^{1}$Department of Mathematics and Information Sciences,}\\
\small{Warsaw University of Technology,}\\
\small{Ul. Koszykowa 75, 00-662 Warsaw, Poland.}\\
\small{$^{2}$Facultad de Ciencias Exactas, Departamento de Matem\'aticas,}\\
\small{Universidad Andres Bello,}\\
\small{Rep\'ublica 498, Santiago, Chile.}}

\begin{document}
\maketitle
\begin{abstract}
We obtain a compact Sobolev embedding for $H$-invariant functions in compact metric-measure spaces, where $H$ is a subgroup of the measure preserving bijections. In Riemannian manifolds, $H$ is a subgroup of the  volume preserving diffeomorphisms: a compact embedding for the critical exponents follows. The results can be viewed as an extension of Sobolev embeddings of functions invariant under isometries in compact manifolds. 
\end{abstract}
\bigskip

\medskip

\section{Introduction}

Arising in the Calculus of Variations and PDE's,
the study of Sobolev spaces in Euclidean domains, and the embeddings between them, has been an active area of research for more than a century (see \cite{Adams} for the classical results, and \cite{naumann} for an overview on history). In the last fifty years, motivated by problems in Geometric Analysis, Physics and Topology, those studies have been generalized to functions on manifolds, with the extension to sections of vector bundles over those spaces, see \cite{aubin, Hawking, hebey, lawson, palais}. More recently, the study of metric-measure spaces\footnote{See \cite{gromov} for an interesting perspective.} demands, whenever it is possible, similar studies in this context, see the books \cite{Ambrosio, Heinonen, H}. \\

A fundamental ingredient in Sobolev spaces is the the concept of weak or distributional gradient. In metric-measure spaces there are at least two notions that provide a valid generalization of the usual gradient in $\mathbb{R}^n$: 
\begin{itemize}
\item An upper gradient, see \cite{Cheeger, H}.
\item The one used in this work, nowadays called a Haj\l{}asz gradient, see \cite{Ambrosio, Hajlasz} and Section \ref{Preliminaries} below.
\end{itemize}
Both notions of gradient have advantages and disadvantages with respect to each other, see \cite{Heinonen, H, JSYY}.\\

If $X$ is a set, $\mu$ is a measure on $X$, and $1 \leq p < \infty$, denote by $L^p(X,\mu)$ the vector space of $\mu$-measurable functions such that
\[
\| f \|_{L^p(X,\mu)} := (\ \int_X | f(x) |^p d \mu(x)\ )^{1/p}
\] 
is finite. In particular, if $X$ is either a bounded domain in $\mathbb{R}^n$ or a compact Riemannian $n$-manifold, with $\mu = V_g$  the volume measure associated to the Euclidean or Riemannian line element $g$, respectively, $L^{p}_1(X,\mu)$ refers to the subspace of $L^p(X,\mu)$ made up of those functions  
such that the norm of their distributional gradient (with respect to $g$) is in $L^p(X,\mu)$. In those cases
one has the embedding 
\[ L^{p}_1(X,\mu) \hookrightarrow L^q(X,\mu) 
\] 
whenever $p \leq q \leq p^{\ast} := n p / (n-p)$, where $1 \leq p < n$. If $q < p^{\ast}$, the embedding is compact, and one writes
\begin{equation}   \label{comp-emb-1}
L^{p}_1(X,\mu) \hookrightarrow \hookrightarrow L^q(X,\mu) ,
\end{equation}
see \cite{Adams, aubin, hebey}.
The non-compactness of the embedding in the limit case $q = p^{\ast}$
is a phenomena that arises from sequences of transformations that induce substantial changes in the functions, transformations that nonetheless leave the norm of functions unchanged. With such information, it is tempting to look for subspaces of $L^{p}_1(X,\mu)$ whose elements are invariant under an appropriate subgroup of 
$\text{Diff}(X)$, and see if the compact embedding (\ref{comp-emb-1}), when restricted to these subspaces, can be extended to higher values of $q$: let $H$ be a subgroup of $\text{Diff}(X)$, and denote by 
$L^{p}_{1,H}(X,\mu)$ the subspace of $L^{p}_1(X,\mu)$ made up
of $H$-invariant functions.\\

The best result in this context is due to E. Hebey and M. Vaugon, who consider $H$ as being a subgroup of $\text{Isom}(X,g)$, the group of isometries of $(X,g)$:


\begin{theorem} (Hebey-Vaugon \cite{hebey-vaugon}, also Theorem 9.1 in \cite{hebey})   \label{heb-vaug}
Suppose $(X,g)$ is a compact Riemannian $n$-manifold, and $H$ is  a compact subgroup of $\text{Isom}(X,g)$. If $H(x)$ denotes the orbit of the point $x$ under the action of $H$, require that $H(x)$ is uncountable for every $x$ in $X$. If  
$k := \text{min}\ \{\ \text{dim}\ H(x) :  x\in X\    \}$, then
\begin{equation}    \label{comp-emb-2}
L^{p}_{1,H}(X,V_g) \hookrightarrow \hookrightarrow L^q(X,V_g)
\end{equation}
whenever
$1 \leq p < n-k$ and $1 \leq q < \frac{(n-k) p}{n-k-p}$.
\end{theorem}

In a metric-measure space $(X,d,\mu)$ conditions for the metric $d$ and the measure $\mu$ are sometimes required, leading to  \textit{synthetic} extensions of \textit{analytic} Riemannian concepts, like curvature, volume, and dimension (see \cite{villani} for a friendly introduction to these ideas). We will use the doubling condition for the metric space $(X,d)$ and the lower Ahlfors $s$-regularity of the metric-measure space 
$(X, d, \mu)$.\footnote{See Section \ref{Preliminaries} for definitions.} For instance, if $(X,g)$ is a compact  $n$-manifold with induced distance $d_g$, then $(X,d_g,V_g)$ is
a lower $n$-regular metric-measure space, and $(X,d_g)$ is doubling. \\

As aforementioned,  we use Haj\l{}asz gradients: denote by $M^p_1(X,\mu)$ the vector space of functions in $L^p(X,\mu)$ such that their Haj\l{}asz gradient is also in
$L^p(X,\mu)$. In Section \ref{examples}, Theorem \ref{riem-mm},  we  see that when $(X,d_g, V_g)$ is the natural metric-measure space induced from a compact  $n$-manifold $(X,g)$, then $M^p_1(X, V_g)$ coincides with
$L^p_1(X, V_g)$ for $1 < p < \infty$.  \\

In the \textit{analytic} context of Riemannian geometry, symmetry groups are subgroups of 
$\text{Diff}(X)$. Instead, in the \textit{synthetic} context of metric-measure spaces, symmetry groups are subgroups of $\text{Aut}(X)$, the group of automorphisms or bijections of $X$: let $H$ be a subgroup of $\text{Aut}_{\mu}(X)$, the group of 
$\mu$-preserving automorphisms of $X$; denote by 
$M^{p}_{1,H}(X,\mu)$ the subspace of $M^{p}_1(X,\mu)$ made up
of $H$-invariant functions. The main result in this work is: 


\begin{theorem}  \label{main2}
 Assume that $(X, d, \mu)$ is a metric-measure space that is compact, Ahlfors lower  $s$-regular, with $(X,d)$ doubling, and such that $M^p_1(X,\mu)$ is reflexive.
If $H$ is a subgroup of  $\text{Aut}_{\mu}(X)$ such that for every $x$ in $X$ the set $H(x)$ is uncountable, then
\begin{equation}    \label{comp-emb-3}
M^p_{1,H} (X,\mu) \hookrightarrow \hookrightarrow L^{q} (X,\mu)
\end{equation}
whenever $1 < p < s$ and $1 \leq q \leq p^{*} = \frac{s p}{s - p}$.
\end{theorem}

\begin{remark}
In contrast with classical Sobolev spaces, there are situations where $M^p_1(X,\mu)$ is not reflexive for $1<p<\infty$: some examples of  this unexpected phenomena are self similar Cantor sets, see \cite{rissanen}. On the other hand, a discussion about sufficient conditions on $(X,d,\mu)$ for $M^p_1(X,\mu)$ to be reflexive can be found in \cite{gorka1, Hajlasz1}.\footnote{For instance, $M^p_1(X,V_g)$ is reflexive for every compact $(X,g)$.} 
\end{remark}


To highlight the contributions of this work, we make some remarks:
\begin{enumerate}
\item Concerning the groups appearing in Theorems \ref{heb-vaug} and \ref{main2}: 
In the context of metric-mesure spaces arising from Riemannian manifolds, the group $H$ in Theorem \ref{main2} is a subgroup of $\text{Diff}_{V_g}(X)$, the group of volume preserving
diffeomorphisms of $(X,g)$; in Theorem \ref{heb-vaug} the group $H$ is a compact subgroup of the smaller group
$\text{Isom}(X,g)$. A classical result of S. Myers and N. Steenrod, see \cite{kobayashi}, provides 
$\text{Isom}(X,g)$ with the structure of a finite dimensional Lie group, that is compact  if $X$ is compact. In contrast, if $X$ is compact, H. Omori provided the larger group 
$\text{Diff}_{V_g}(X)$, and $\text{Diff}(X)$ as well, with the structure  of an Inverse Limit of Hilbert manifolds, see \cite{ebin, kobayashi}. The Lie algebras of both groups are \textit{represented} by vector fields: those in the \textit{formal}
 Lie algebra of 
$\text{Diff}_{V_g}(X)$ are free of divergence; those in the Lie algebra of $\text{Isom}(X,g)$ are Killing, a stronger condition. In every Riemannian manifold there are non-trivial vector fields free of divergence; on the other hand, the sign of the Ricci curvature imposes restrictions for Killing vector fields: if the Ricci tensor is non-positive and negative definite at some point, there are no non-trivial Killing vector fields, and the group 
$\text{Isom}(X,g)$ is finite \cite{Berard, kobayashi}.

\item Concerning the proofs of Theorems \ref{heb-vaug} and \ref{main2}: Roughly speaking, Theorem \ref{heb-vaug} uses local charts compatible with the dimension reduction under the Riemannian submersion induced by the isometries, reducing the compact embedding of functions to a Sobolev inequality in the space orthogonal to the $H$-orbits, providing a convenient  setup for specific results obtained by H. Berestycki, E. Lieb, P. L. Lions and others, see \cite{lions1}. The proof of Theorem \ref{main2} is different: the dimension reduction compatible with isometries used in Theorem \ref{heb-vaug} is not always compatible with volume preserving diffeomorphisms.\footnote{The quotient space might not be Hausdorff.} Some ingredients in the proof are a Sobolev-Haj\l{}asz inequality \cite{Hajlasz}, and variations of the Concentration-Compactness principle, see \cite{lions2}.\\
\end{enumerate}
In Section \ref{Preliminaries} we provide definitions and results that will be used in Section \ref{main}, where a detailed proof of Theorem \ref{main2} is given. In  
Section \ref{examples} we see that Theorem \ref{main2} can be applied
in the Riemannian context, and discuss necessary and sufficient conditions for its applicability when the dimension of the $H$-orbits is 1.\\

For Sobolev embeddings in non-compact spaces using symmetry, see  \cite{Gaczkowski, gorka1}, and the references there.


 \section{Preliminaries} \label{Preliminaries}

In this work $(X, d, \mu)$ is a metric-measure space equipped with a metric $d$ and a Radon measure $\mu$. 
We assume  that the measure of every open non-empty set is positive, and that the measure of every bounded set is finite. 

In most parts of our paper we will assume that the metric-measure space $(X, d, \mu)$ is \textit{lower Ahlfors $s$-regular}: this means that 
there exists a constant $b$ such that
\[
   b R^s \leq \mu \left(B_R(x)\right) 
\]
for all balls $B_R(x)$ in $X$ with $R < \hbox{diam} X$.

A metric space is said to be \textit{doubling} if there exists a constant $C$ such that for every ball of radius $R$, there 
exist $C$ balls of radius $R/2$ that cover the original ball. It not difficult to see that if  $(X, d, \mu)$ is a doubling metric-measure space,\footnote{A metric-measure space $(X,d,\mu)$ is said to be \textit{doubling} if the measure $\mu$ is \textit{doubling}, namely  if 
there exists a constant $C_{\mu}>1$ such that for every ball $B_R(x)$
one has
$ \mu \left(B_{2R}(x)\right) \leq C_{\mu}\ \mu \left(B_R(x)\right).$
} then $(X, d)$ is a doubling metric space (see \cite{gromov}, Appendix $B_+$). Conversely, J. Luukkainen and E. Saksman in \cite{Luukkainen} prove that every complete doubling metric space carries a doubling measure. \\

If $(X, d, \mu)$ is a metric-measure space,
say that a  function $f$ in $L^p(X,\mu)$ belongs to the \textit{Haj{\l}asz-Sobolev} space  $M^{p}_{1}(X,\mu)$ if there exists some $g \in L^{p}(X,\mu)$, called  a \textit{Haj\l{}asz gradient}, such that 
\begin{eqnarray}   \label{global-local}
|f(x) - f(y)| \leq  d(x,y) \left( g(x) + g(y) \right) 
\end{eqnarray}
for $\mu$ almost every $x$ and $y$ in $X$. In this context, given $f$ in $M^{p}_1(X,\mu)$, we denote by $g_f$ any Haj\l{}asz gradient for $f$, to endow  the space $M^p_1(X,\mu)$ with the norm
\begin{eqnarray}   \label{global-norm}
  \|f\|_{M^p_1(X,\mu)}:= \|f\|_{L^{p}(X,\mu)}+\inf_{g_f} \|g_f\|_{L^{p}(X,\mu)},
\end{eqnarray}
and then $M^p_1(X,\mu)$ is a Banach space. 

In the same context, say that $f \in L^p(X,\mu)$ belongs to $m^p_1(X,\mu)$ if there exists some $g \in L^{p}(X,\mu)$, called a \textit{local Haj\l{}asz gradient}, such that for every $z$ in $X$ there exists an open set $U_z$ and some $E_z \subset U_z$ with $\mu(E_z) = 0$, such that for every pair of points $\{x,y\}$ in $U_z  \thicksim E_z$ the inequality (\ref{global-local}) holds. As in (\ref{global-norm}) one defines
\begin{eqnarray}   \label{local-norm}
  \|f\|_{m^p_1(X,\mu)}:= \|f\|_{L^{p}(X,\mu)}+\inf_{g_f} \|g_f\|_{L^{p}(X,\mu)},
\end{eqnarray}
where now the infimum is over all those $g_f$ that are local Haj\l{}asz gradients for $f$. Then $m^p_1(X,\mu)$ is also a Banach space.
It is obvious that  Haj\l{}asz gradients for a function $f$ are local Haj\l{}asz gradients; the converse is not true in general, see \cite{JSYY} for an example.

For a detailed exposition of some basic properties of these spaces, we refer to \cite{Ambrosio, Hajlasz, Hajlasz1, HajlaszKoskela,  Heinonen, H, JSYY, Kinnunen}. \\

The next result will be useful:
 \begin{prop} (See \cite{Hajlasz}) \label{wlozenie}
Suppose $(X, d, \mu)$ is an Ahlfors lower $s$-regular metric-measure space of finite diameter. If $1 < p < s$, then
	\begin{eqnarray*}
		M^p_1(X,\mu)  \hookrightarrow L^{p^*}(X,\mu),
	\end{eqnarray*}
 where $p^* =\frac{sp}{s-p}$. Moreover, there exists  a constant  
$C=C(s,p,b)$, depending on $s, p, b$, such that for each $f$ in $M^p_1(X,\mu)$ 
\begin{eqnarray*}
		\|f\|_{L^{p^*}(X,\mu)} \leq C\left(\|f\|_{L^{p}(X,\mu)}+\| g_f\|_{L^{p}(X,\mu)}\right)
	\end{eqnarray*}
whenever $g_f$ is a Haj\l{}asz gradient for $f$.
\end{prop}

We use Proposition \ref{wlozenie} to infer:

 \begin{prop} \label{noncritic}
Assume that $(X, d, \mu)$ is an Ahlfors lower $s$-regular compact metric-measure space, with $(X, d)$  doubling. If $1 < p < s$, then for every $q <p^*$
	\begin{eqnarray*}
		M^p_1(X,\mu)  \hookrightarrow \hookrightarrow L^q(X,\mu),
	\end{eqnarray*}
  where $p^* =\frac{sp}{s-p}$.
\end{prop}
\begin{proof}
By Proposition \ref{wlozenie} have that $M^p_1(X,\mu) \hookrightarrow L^q(X,\mu)$ for every $q \in [1, p^*]$. Moreover, 
since $(X, d)$ is doubling, by Theorem 2 in \cite{Kalamajska} we have the compact embedding
\begin{eqnarray}\label{comp}
  M^p_1(X,\mu) \hookrightarrow \hookrightarrow L^{p}(X,\mu).
\end{eqnarray}
Hence if $q\in [1,p]$, then
\[
	M^p_1(X,\mu) \hookrightarrow \hookrightarrow L^p(X,\mu)  \hookrightarrow L^q(X,\mu).
\]
Next, consider the case when $p < q < p^*$. We will prove that the ball $\mathfrak{B}=\{f: \|f\|_{M^p_1(X,\mu)}\leq 1\}$ is precompact in $L^q(X,\mu)$. Fix $\theta$ in $(0,1)$ so that
\[
	\frac{1}{q}=\frac{\theta}{p}+\frac{1-\theta}{p^*},
\]
and use (\ref{comp}) to note that $\mathfrak{B}$ is precompact in $L^p(X,\mu)$. Hence for every $\epsilon >0$ there exists an 
$\tilde{\varepsilon}=2C\epsilon^{\frac{1}{\theta}} / (2C)^{\frac{1}{\theta}}$ net\footnote{This means that for each $f$ in $\mathfrak{B}$ there exists some $k$ in $\{1,...,N\}$ such that $\|f-f_k\|_{L^p(X,\mu)} <\tilde{\varepsilon}$.} of $\mathfrak{B}$ in $L^p(X,\mu)$, say 
 $\{f_k\}_{k \in \{1,...,N\}}$, where
$C$ is the constant from Proposition \ref{wlozenie}. Now it is enough to prove that $\{f_k\}_{k \in \{1,...,N\}}$ is an $\epsilon$ net of $\mathfrak{B}$ in $L^q(X,\mu)$; indeed, by the interpolation inequality we have
\begin{eqnarray*}
		\|f-f_k\|_{L^q(X,\mu)}&\leq&\|f-f_k\|^{\theta}_{L^p(X,\mu)}\|f-f_k\|^{1-\theta}_{L^{p^*}(X,\mu)}\\
		&\leq&C^{1-\theta}\|f-f_k\|^{\theta}_{L^p(X,\mu)}\|f-f_k\|^{1-\theta}_{M^p_1(X,\mu)}\\
		&\leq&2^{1-\theta}C^{1-\theta}\|f-f_k\|^{\theta}_{L^p(X,\mu)}\leq \epsilon
\end{eqnarray*}
for some $k$ in $\{1,...,N\}$.
\end{proof}

\begin{remark}
Proposition \ref{noncritic} highlights the fact that in general one cannot expect that $M^p_1(X,\mu)  \hookrightarrow \hookrightarrow L^{p^*}(X,\mu)$. Theorem \ref{main2} ensures that some proper subset of $M^p_1(X,\mu)$ is relatively compact in  $L^{p^*}(X,\mu)$.
\end{remark}


 \subsection{Auxiliary Lemmata}

The next lemma seems to be well known, however we give a detailed proof due to its role in Section \ref{main}.
\begin{lemma} \label{lemma:delty}
Here $(\Omega ,d , \mu)$ is a separable metric-measure space with a finite Borel measure $\mu$. 
Suppose that there exists some $\delta >0$ such that for every measurable set $A$, either $\mu(A) = 0$ or $\mu(A) \geq \delta$. Then 
there exists a finite set $\{x_i \}_{i \in I}$ of points in $\Omega$ and a finite set of numbers $\{ \mu_i \}_{i \in I}$ not smaller than $\delta$ such that
\[
\mu = \sum_{i \in I} \mu_i \delta_{x_i}.
\]
\end{lemma}

\begin{proof}
Consider the  set $A_{\delta} := \left\{ x \in \Omega \, : \, \mu(x) \geq \delta \right\}$. Since $\mu(\Omega) < \infty$, the set 
$A_{\delta}$ must be finite.
We will show that $\mu(\Omega \thicksim A_{\delta}) = 0$.  Since $\Omega \thicksim A_{\delta}$ is open, we have
\[
\Omega \thicksim A_{\delta} = \bigcup_{x \in \Omega \thicksim A_{\delta}} B_{R_x}(x) .
\]
Moreover, since there are no atoms in 
$\Omega \thicksim A_{\delta}$, for every $x$ in $\Omega \thicksim A_{\delta}$ we can choose $R_x$ in such a way that
\[
\mu(B_{R_x}(x)) = 0.
\]
Furthermore, since $\Omega$ is separable, Lindel\"{o}f's lemma yields
\[
\Omega \thicksim A_{\delta} = \bigcup_{x \in A} B_{R_x}(x),
\]
where $A$ is a countable subset of  $\Omega \thicksim A_{\delta} $, and $\mu( \Omega \thicksim A_{\delta} ) = 0$ follows.
\end{proof}

To state the next lemma, given some space $F(\Omega)$ of functions on some set $\Omega$, denote by
\[
F_c (\Omega) :=\{\phi \in F(\Omega): \hbox{spt} \phi  \subset \subset \Omega \}
\]
the subset of  $F(\Omega)$  consisting of those functions whose support is a compact subset of
$\Omega$.
 
\begin{lemma} \label{aproksymacja}
Here $(X,d)$ is a locally compact metric space with two Radon measures 
$\mu$ and $\nu$, and  $ \Omega \subset X$ is a precompact open set. Then for every $p$ and $r$ in $[1, \infty)$ the set 
$\hbox{Lip}_c (\Omega)$ is equidense both in $L^{r}(\Omega,\mu)$ and in $L^{p}(\Omega,\nu)$. This means that for every $\epsilon >0$ and every $f \in L^{r}(\Omega,\mu) \cap L^{p}(\Omega,\nu)$ 
there exists some $\phi$ in $\hbox{Lip}_c (\Omega)$ such that
\[
  \|f -\phi\|_{L^r(\Omega,\mu)} \leq \epsilon \quad \text{and} \quad \|f-\phi\|_{L^p(\Omega,\nu)} \leq \epsilon.
\]
\end{lemma}

\begin{remark}
By Urysohn's lemma $C_c (\Omega)$ is dense in $L^r(\Omega,\mu)$ and in $L^p(\Omega,\nu)$.
\end{remark}
\begin{proof}
To prove the lemma it is sufficient to check that for every measurable set $A$ the characteristic function $\mathbf{1}_A$ 
can be approximated both in $L^r(\Omega,\mu)$ and in 
$L^p(\Omega,\nu)$  by functions in $\hbox{Lip}_c (\Omega)$. 
The regularity of the measures ensures that there exists a sequence $\{K_n\}_n$ of compact sets and a sequence $\{U_n\}_n$ of open sets such that
$ K_n \subset A \subset U_n $, with
\[ 
\mu(U_n \thicksim K_n) \leq \frac{1}{n}\   \text{and} \ 
\nu(U_n \thicksim K_n) \leq \frac{1}{n}.
\]
Without loss of generality we 
can assume that $U_n \subset \Omega$. Moreover, since the space is locally compact, for every $n$
there exists an open precompact set $V_n$ such that 
\[
K_n \subset V_n \subset \overline{V}_n \subset U_n.
\]
Introduce the sequence of functions $\{\psi_n\}_n$ given for each $n$ by
\[
\psi_n := \mathbf{1}_{K_n} : K_n \cup (\Omega \thicksim V_n)\rightarrow [0,1] ,
\]
and denote by $\tilde{\psi}_n$ the extension of $\psi_n$ to all $\Omega$ defined as
\[
	\tilde{\psi}_n(x) := \sup_{y \in  K_n \cup (\Omega \thicksim V_n) } \left(  \psi_n(y) - L_n\ d(x,y) \right) ,
\]
where $L_n = 1 / \text{ dist } ( K_n , \Omega \thicksim V_n)$. Such functions are Lipschitz on $\Omega$, with $\tilde{\psi}_n = \psi_n$ on $ K_n \cup (\Omega \thicksim V_n)$, and with $\tilde{\psi}_n \leq 1$.
Finally, define 
\[
\phi_n = \max \{0, \tilde{\psi}_n \},
\]
and note that $\phi_n \in \hbox{Lip}_c (\Omega)$. 
Then it is easy to see that 
\[
\int_{\Omega} \left| \phi_n(x) - \mathbf{1}_A(x) \right|^{r} d \mu(x) \leq 2^{r} \mu(U_n \thicksim K_n) \leq \frac{2^{r}}{n} ,
\]
and similarly
\[
\int_{\Omega} \left| \phi_n(x) - \mathbf{1}_A(x) \right|^{p} d \nu(x) \leq \frac{2^{p}}{n}.
\]
\end{proof}


\section{Proof of Theorem \ref{main2}} \label{main}

In this section we prove  Theorem \ref{main2}, our main result. To prove such a theorem, we will need Theorem \ref{theorem:Lions}, which in turn requires Lemma \ref{lemma:rhol}, Lemma \ref{prod} and Lemma \ref{BL}.  We start with  Lemma \ref{lemma:rhol}:

\begin{lemma} \label{lemma:rhol}  (Reverse H\"older)
Let $\Omega \subset X$ be an open precompact subset of the metric space $(X, d)$, and let $\mu$ and $\nu$ be Radon measures on $\Omega$. Assume 
that $1 \leq p < r$.
If there exists a positve real number $C$ such that for every Lipshitz $\phi$ with compact support
\begin{equation}    \label{hypothesis}
\| \phi \|_{L^r(\Omega,\mu)} \leq C \| \phi \|_{L^{p}(\Omega,\nu)},
\end{equation}
then there exists a countable set of points $\{x_i \}_{i \in I}$ in $\Omega$ and a countable set $\{\mu_i\}_{i \in I}$ of  positive real numbers  such that
\[
\mu = \sum_{i \in I} \mu_i \delta_{x_i}.
\]
\end{lemma}

\begin{proof} 
We divide the proof into two steps.\\

{\bf Step 1.} Assume that $\mu = \nu$, choose any measurable set $A$, and assume that (\ref{hypothesis}) holds;  by Lemma \ref{aproksymacja}
\[
  \| \mathbf{1}_{A} \|_{L^{r}(\Omega,\mu)} \leq C \| \mathbf{1}_{A} \|_{L^{p}(\Omega,\mu)},
\]
and then
\[
\mu(A)^{\frac{1}{r}} = \| \mathbf{1}_{A} \|_{L^r(\Omega,\mu)} \leq C \| \mathbf{1}_{A} \|_{L^p(\Omega,\mu)} = C \mu(A)^{\frac{1}{p}}.
\]
Hence either $\mu(A) = 0$, or
$\mu(A) \geq 1 / {C}^{\frac{pr}{r- p}}$.
Then by Lemma \ref{lemma:delty} there exists finite set $\{x_i \}_{i \in I}$ of points in $\Omega$ and  a finite set $\{ \mu_i \}_{i \in I}$ of real numbers with $\mu_i \geq  1/C^{\frac{pr}{r- p}}$ such that
\[
\mu = \sum_{i \in I} \mu_i \delta_{x_i}.
\]

{\bf Step 2.} Now assume that $\mu$ and $\nu$ are arbitrary; the Lebesgue Decomposition theorem ensures that
\begin{equation}    \label{decomposition}
\nu = \mu \llcorner \theta + \sigma
\end{equation}
for some non-negative $\theta$ in $L^1(\Omega,\mu)$, where 
$\mu \llcorner \theta (A) := \int_{A} \theta d \mu$, and $\sigma$ is a positive measure singular with respect to $\mu$. For every positive integer $n$ consider
the function 
\[
\phi_n := \theta^{\frac{1}{r - p}} \mathbf{1}_{\theta \leq n} \psi   ,
\]
where $ \psi $ is  Lipschitz  with compact support, and also the measure 
\[
\mu_n :=  \mu \llcorner ( \theta^{\frac{r}{r - p}} \mathbf{1}_{\theta \leq n} ).
\]

Assuming (\ref{hypothesis}) and recalling Lemma \ref{aproksymacja}, use the decomposition (\ref{decomposition}) to obtain
\begin{equation}   \label{ineq-rp}
\| \phi_n \|_{L^r(\Omega,\mu)} \leq C \| \phi_n \|_{L^p(\Omega,\nu)} = C \| \phi_n \|_{L^p(\Omega,  \mu \llcorner \theta + \sigma)} 
=  C \| \phi_n \|_{L^p(\Omega,  \mu \llcorner \theta )} . 
\end{equation}
However
\begin{equation}     \label{eq-p}
\| \phi_n \|^p_{L^p(\Omega,  \mu \llcorner \theta )} =
\int_{\Omega} \left| \psi \right|^p \theta^{\frac{p}{r-p}} \mathbf{1}_{\theta \leq n} \theta d \mu = 
\int_{\Omega} \left| \psi \right|^p \theta^{\frac{r}{r-p}} \mathbf{1}_{\theta \leq n} d \mu =
\| \psi \|^p_{L^p(\Omega , \mu_n)} ,
\end{equation}
and similarly
\begin{equation}    \label{eq-r}
\| \psi \|_{L^r(\Omega , \mu_n)} = \| \phi_n \|_{L^r(\Omega, \mu)} .
\end{equation}

Then use (\ref{eq-p}) and (\ref{eq-r}) in (\ref{ineq-rp})  to infer that
\[
\| \psi \|_{L^r(\Omega , \mu_n)} \leq C \| \psi \|_{L^p(\Omega ,  \mu_n)} 
\]
for every $n$.\\

Hence by Step 1 
\[
\mu_n = \sum_{i \in I_n} \mu_{n,i} \delta_{x_{n,i}}
\]
for every $n$. Recall the definition of the measures $\mu_n$, and 
note that $\hbox{spt} \mu_n \subset \hbox{spt} \mu_{n+1}$, in particular $I_n \subset I_{n+1}$. Let $I=\bigcup_{n=1}^{\infty}I_n$ and define $x_i : =x_{n,i} \big|_{I_n}$; one can write
\[
\mu_n = \sum_{i \in I_n} \mu_{n,i} \delta_{x_i}.
\]
Since $\mu_{n,i} =\mu_n (\{x_i\})\leq  \mu_{n+1} (\{x_i\})=\mu_{n+1,i} $, it follows that for each $i$ the number $\mu_{n,i}$ is non decreasing with respect to $n$.\\ 

Denote by $\mathcal{M}(\Omega)$ the set of measures on $\Omega$ endowed with the weak-$\ast$ topology. Let $\tilde{\mu}_n= \mu_n \llcorner (\theta^{-\frac{r}{r-p}}\mathbf{1}_{\{\theta>0\}})$, and observe that
 $\tilde{\mu}_n \rightarrow \mu \llcorner \mathbf{1}_{\{\theta>0\}}$
in $\mathcal{M}(\Omega)$. We claim that 
\[
 \tilde{\mu}_n \rightarrow  \mu .
\]
To prove that, it suffices to show that $\mu(\{\theta=0\})=0$. Since $\mu$ is singular with respect to $\sigma$, there exist subsets $A$ and  $B$ with $A \cap B =\emptyset$, such that for every measurable set $E$ we have $\mu(E)= \mu (A\cap E)$ and $\sigma(E)= \sigma (B\cap E)$. Therefore
\[
	\int_{\Omega}\mathbf{1}_A \mathbf{1}_{\{\theta=0\}} d \nu =\int_{\Omega}\mathbf{1}_A \mathbf{1}_{\{\theta=0\}} \theta d \mu + \int_{\Omega}\mathbf{1}_A \mathbf{1}_{\{\theta=0\}} d \sigma=0,
\]
hence $\nu (A\cap \{\theta=0\})=0$, and using (\ref{hypothesis})
\[
\| \mathbf{1}_{A\cap \{\theta=0\}} \|_{L^r(\Omega,\mu)} \leq C \| \mathbf{1}_{A\cap \{\theta=0\}} \|_{L^{p}(\Omega,\nu)}=0.
\]
Thus $\mu ( \{\theta=0\})=\mu (A\cap \{\theta=0\})=0$, as required.
\end{proof}

Now we continue with Lemma \ref{prod}:

\begin{lemma} (Haj\l{}asz-Leibniz) \label{prod}
If $v \in M^p_1(X,\mu)$ and $\phi \in \hbox{Lip}(X)$, then $f = v \phi \in M^p_1(X,\mu)$. Moreover,
 \[
g_f = g_v |\phi| + |v|\|\phi\|_{\hbox{Lip}}
 \]
is a Haj\l{}asz gradient for $v \phi$.
\end{lemma}
\begin{proof}
The result follows from the string of inequalities
 \begin{eqnarray*}
  |v(x)\phi(x) - v(y)\phi(y)| &\leq& |v(x)-v(y)|\ \min\{|\phi(x)|,|\phi(y)|\} + \max \{|v(x)|,|v(y)| \}\ |\phi(x)-\phi(y)|\\
  &\leq& \big( g_v(x)+g_v(y) \big)\  \min \{ |\phi(x)|,|\phi(y)|\}\ 
d(x,y) + \max \{ |v(x)|,|v(y)| \}\ \|\phi\|_{\hbox{Lip}}\ d(x,y)\\
  &\leq& \big( g_v(x)|\phi(x)|+g_v(y)|\phi(y)| \big)\ d(x,y) + (|v(x)|+|v(y)|)\ \|\phi\|_{\hbox{Lip}}\ d(x,y) \\
& = & d(x,y)\ ( g_{v \phi}(x) + g_{v \phi}(y) ),
\end{eqnarray*}
with  $g_{v \phi} := g_v |\phi| + |v|\|\phi\|_{\hbox{Lip}} .$
\end{proof}


Finally, before stating Theorem \ref{theorem:Lions}, we recall the following Lemma attributed to H. Br\'ezis and E. Lieb:

\begin{lemma} \label{BL}
Let $p \in [1, \infty)$. If $f_n \rightarrow f$ weakly in $L^{p}(X,\mu)$ and $f_n \rightarrow f$ $\mu$-almost everywhere, then
\[
\lim_{n \rightarrow \infty} \left(  \int_{X}  |f_n|^{p} d\mu - \int_{X}  |f_n - f|^{p} d \mu \right) = \int_{X}  |f|^{p} d \mu.
\]
\end{lemma}

With those results at hand, we have:


\begin{theorem} \label{theorem:Lions}
If $(X, d, \mu)$ is an Ahlfors lower $s$-regular compact metric-measure space, and $1 < p < s$,
then for every sequence $\left\{ u_n \right\}$ in $M^p_1(X,\mu)$ such that
$u_n \rightarrow u$ weakly in $M^p_1(X,\mu)$ and $u_n \rightarrow u$ strongly in $L^p(X,\mu)$, there exists a 
subsequence $\left\{ u_n \right\}$ and a countable set $I$ such that
\begin{equation} \label{eq:form}
\mu \llcorner |u_n|^{p^*} \to \mu \llcorner |u|^{p^*}  + \sum_{i \in I} \mu_i \delta_{x_i}
\end{equation}
in $\mathcal{M} (X)$, where $x_i \in X$ for every $i \in I$.
\end{theorem} 

\begin{proof}

We begin with two observations:
\begin{enumerate}
\item Let $v_n:= u_n - u$, and fix some $\phi$ in $\hbox{Lip}_c(X)$. The hypotheses, Proposition \ref{wlozenie} and Lemma \ref{prod} when applied to $f_n := v_n \phi$  give
\begin{equation}  \label{istot}
\|v_n \phi\|_{L^{p^*}(X,\mu)} \leq C\left(\|v_n \phi\|_{L^{p}(X,\mu)}+\| g_{v_n} \phi\|_{L^{p}(X,\mu)} + 
\|\phi\|_{\hbox{Lip}}  \| v_n \|_{L^{p}(X,\mu)}\right) .
\end{equation}

\item The hypotheses also imply that
\begin{itemize} 
\item $\|v_n \phi\|_{L^{p}(X,\mu)} \to 0$ and $\| v_n \|_{L^{p}(X,\mu)}  \to 0$, 
\item $\mu \llcorner  |v_n|^{p^{\ast}} \to \bar{\mu}$ and $\mu \llcorner |g_{v_n}|^{p} \to \nu$ for some $\bar{\mu}$ and $\nu$ in 
$\mathcal{M}(X)$.
\end{itemize}
\end{enumerate}

Those observations yield the reverse H\"older inequality
\[
\| \phi \|_{L^{p^*}(X,\bar{\mu})} \leq C \| \phi \|_{L^{p}(X,\nu)} ,
\]
and Lemma \ref{lemma:rhol} ensures that the mesure $\bar{\mu}$ has the form
\begin{equation} \label{eq:suma}
\bar{\mu} = \sum_{i \in I} \mu_i \delta_{x_i}.
\end{equation}

Now use Lemma \ref{BL} when $f_n=u_n \phi^{\frac{1}{p^*}}$, where $\phi$ is a non-negative function in $C_c(X)$, and
\[
\lim_{n \rightarrow \infty} \left(  \int_{X} \phi |u_n|^{p^{*}} d \mu - \int_{X}  \phi  |v_n|^{p^{*}} d \mu \right) =
\int_{X}  \phi  |u|^{p^{*}} d \mu
\]
follows. Then recall that $\mu \llcorner |v_n|^{p^{*}} \to \bar{\mu}$ in 
$\mathcal{M}(X)$, to infer
\begin{equation} \label{eq:pos}
\lim_{n \rightarrow \infty} \int_X \phi |u_n|^{p^{*}} d \mu  =  \int_X \phi\   d \bar{\mu} + \int_X \phi |u|^{p^{*}}  d \mu.
\end{equation}

Since every continuous function of compact support, say $\phi$, can be written as $\phi=\phi_{+}-\phi_{-}$, where $\phi_{+}$ and $\phi_{-}$ are non-negative with compact support, one  concludes that
(\ref{eq:pos}) holds for every $\phi$ in the dual of $\mathcal{M}(X)$. Now use (\ref{eq:suma}), to obtain (\ref{eq:form}). 
\end{proof}


Now we can prove  Theorem \ref{main2}, the main result in this work.
\begin{proof}

By the hypotheses, whenever $h \in H$ one has $h_{\#} \mu = \mu$. Let $\{ u_n \}$ be a bounded sequence in $M^p_{1,H}(X,\mu)$, namely a bounded sequence in $M^p_1(X,\mu)$ such that $h^{\#} u_n = u_n$ for each $n$ and each $h$ in $H$.
Then the sequence of measures $\{ \mu_n \}$ defined by
\[
\mu_n := \mu \llcorner |u_n|^{p^{\ast}}
\]
is also $H$-invariant.

On the other hand, if the sequence $\{ u_n \}$ converges weakly to some $u$ in
$M_{1,H}^p(X,\mu)$, then\footnote{$M_1^p(X,\mu)$ is reflexive in the hypotheses of the theorem.} by Theorem \ref{theorem:Lions} there exists a subsequence\footnote{We use the same subindex for sequences and the pertinent subsequences.} of $\{ \mu_n \}$ such that 
\begin{equation}   \label{measure-inv}
\mu_n \to \mu \llcorner |u|^{p^{\ast}} + \sum_{i \in I} \mu_i \delta_{x_i}
\end{equation}
in $\mathcal{M}(X)$, where $I$ is at most countable. 

In addition, it is not difficult to see that if $\{ \mu_n \}$ is a sequence of $H$-invariant measures converging to some $\nu$ in $\mathcal{M}(X)$, then $\nu$ is also $H$-invariant; therefore from (\ref{measure-inv}) the measure
\[
\mu \llcorner |u|^{p^{\ast}} +  \sum_{i \in I} \mu_i \delta_{x_i}
\]
is $H$-invariant. Moreover, since $\mu \llcorner |u|^{p^{\ast}}$ is $H$-invariant, then
$\sigma := \sum_{i \in I} \mu_i \delta_{x_i}$ is $H$-invariant as well. \\

Choose some $k$ in $I$, and let $y = h (x_k)$ be some element in $H(x_k)$. Then
\[
\mu_k = \sigma(x_k) = \sigma(h^{-1}(y)) = h_{\#} \sigma(y) = \sigma(y) = \sum_{i \in I} \mu_i \delta_{x_i}(y) ,
\]
hence $x_i = y$ for some $i \in I$. This gives a contradiction, since $I$
is at most countable, meanwhile the orbit of each point in $X$ is uncountable by hypothesis. It follows that 
\[
\mu \llcorner |u_n|^{p^{\ast}} \to \mu \llcorner |u|^{p^{\ast}}
\]
in $\mathcal{M}(X)$; but this is equivalent to say that
\begin{equation}   \label{M(X)}
\| \phi u_n \|_{L^{p^{\ast}}(X,\mu)} \to \| \phi u \|_{L^{p^{\ast}}(X,\mu)}
\end{equation}
whenever $\phi \in C_c(X)$. \\

Since $X$ is compact, we can use $\phi = 1$ in (\ref{M(X)}), to conclude that if $\{u_n \}$ is a bounded sequence in $M^p_{1,H}(X,\mu)$ converging weakly to some $u$ in $M^p_{1,H}(X,\mu)$, then
\[
\| u_n \|_{L^{p^{\ast}}(X,\mu)} \to \|  u \|_{L^{p^{\ast}}(X,\mu)}
\]
for some subsequence. But $L^{p^{\ast}}(X,\mu)$ is uniformly convex, hence $u_n \to u$ in $L^{p^{\ast}}(X,\mu)$.
\end{proof}

A useful consequence of Theorem \ref{main2} is:

\begin{corol}   \label{corollary}
Using the same hypotheses as in Theorem \ref{main2}, define the constant $C$ by
\[
C:= \inf \{\ A > 0\ :\ \| u \|_{L^{p^{\ast}}(X,\mu)} \leq A \| u \|_{M^p_1(X,\mu)}\ \text{whenever}\ u \in  M^p_{1,H}(X,\mu)\      \}.
\]
Then there exists some $u_0$ in $M^p_{1,H}(X,\mu)$ such that
\[
C = \| u_0 \|_{L^{p^{\ast}}(X,\mu)}\ /\ \| u_0 \|_{M^p_{1,H}(X,\mu)} .
\]
\end{corol}

\begin{proof}
Define the functional 
$\mathcal{I} :  M^p_{1,H}(X,\mu) \thicksim  \{ 0\} \to  \mathbb{R}$ by
\[
\mathcal{I}[u] := \|  u \|_{M^p_1(X,\mu)} ,
\]
and set
\[
D := \inf \{\ \mathcal{I}[u]\ :\ u \in M^p_{1,H}(X,\mu)\ ,\   \| u \|_{L^{p^{\ast}}(X,\mu)} = 1   \ \} .
\]
Let $\{ u_n \}$ be a minimizing sequence, i.e. such that $u_n \in M^p_{1,H}(X,\mu)$ and $ \| u_n \|_{L^{p^{\ast}}(X,\mu)} = 1$ for every $n$, with $\mathcal{I}[u_n] \to D$. Since $\{ u_n \}$ is bounded in 
$M^p_{1,H}(X,\mu)$, by Theorem \ref{main2} there is a subsequence\footnote{As in Theorem \ref{main2}, we use the same subindex for the sequence and the pertinent subsequence.} of $\{ u_n \}$ and some $w$ in $M^p_{1,H}(X,\mu)$ such that
\[
u_n \to w\ \text{weakly in}\ M^p_{1,H}(X,\mu), 
\] 
\[
\text{and}\ u_n \to w\ \text{strongly in}\ L^{p^{\ast}}(X,\mu).
\]

But $\| w \|_{L^{p^{\ast}}(X,\mu)} = 1$ by strong convergence in $L^{p^{\ast}}(X,\mu)$, hence
\[
D = \lim_{n \to \infty} \mathcal{I}[u_n] = \liminf_{n \to \infty} \|  u_n \|_{M^p_1(X,\mu)} \geq \|  w \|_{M^p_1(X,\mu)} = \mathcal{I}[w] \geq D.
\]

Therefore $\mathcal{I}[w] = D$, and it follows that $C = 1/D$, with $u_0 = w$.
\end{proof}


\section{Riemannian applications}    \label{examples}

The next result is not surprising and probably not new, however we could not find it in the literature. To satisfy the interested reader, and justify the discussion  in Section \ref{final} below, we give a proof of it with some details.

\begin{theorem} \label{riem-mm} Suppose $(X, g)$ is a compact Riemannian $n$-manifold. Then for every $p$ such that $1 < p  < \infty$ the spaces $L^p_1(X, V_g)$ and $M^p_1(X, V_g)$ coincide with equivalent norms.
\end{theorem}

\begin{proof}
By Proposition 10.1 from  \cite{HajlaszKoskela}
\[
M^p_1(X,V_g) \hookrightarrow L^p_1(X,V_g),
\]
hence we need the opposite inclusion. Since $X$ is compact, there exists  a finite number of charts 
\[
\{(U_{\alpha},\phi_{\alpha})\ :\ \alpha \in A \}, 
\]
such that for every $\alpha$ 
the components $g_{i j}^\alpha$ of $g$ in $(U_{\alpha},\phi_{\alpha})$ satisfy 
\begin{eqnarray*}
 \frac{1}{2} \delta_{ij} \leq g^{\alpha}_{ij} \leq 2 \delta_{ij}
\end{eqnarray*}
as  bilinear forms. Let $\{ \eta_{\alpha} \}$ be smooth partition of unity subordinate to covering $\{ U_{\alpha} \}$. We proceed in two steps.\\

{\bf Step 1.} Let $\mathcal{L}^n$ be the $n$-dimensional Lebesgue measure, and fix $u$ in $C^{\infty}(X)$. Since $M^p_1(\mathbb{R}^n, \mathcal{L}^n)$ and $L^p_1(\mathbb{R}^n, \mathcal{L}^n)$ are equivalent, see \cite{Hajlasz} for example, there exists a constant $C > 1$ such that for every $\alpha$ in $A$
\begin{eqnarray}\label{ru}
\frac{1}{C} \| (\eta_{\alpha} u ) \circ \phi_{\alpha}^{-1} \|_{L^p_1(\mathbb{R}^n, \mathcal{L}^n)} \leq 
\| (\eta_{\alpha} u ) \circ \phi_{\alpha}^{-1} \|_{M^p_1(\mathbb{R}^n, \mathcal{L}^n)} \leq C \| (\eta_{\alpha} u ) \circ \phi_{\alpha}^{-1}\|_{L^p_1(\mathbb{R}^n,\mathcal{L}^n)}.
\end{eqnarray}
Furthermore
\[
    \int_X |\eta_{\alpha} u|^p d V_g =  \int_{U_{\alpha}} |\eta_{\alpha} u|^p d V_g  =  \int_{\phi_{\alpha}(U_{\alpha})} \sqrt{ \det g^{\alpha}_{ij}} \left| \eta_{\alpha} u \right|^{p} \circ \phi^{-1}_{\alpha}(x)\ d \mathcal{L}^n(x) ,  
\]
hence
\begin{eqnarray}  \label{niermodularfunkc}
    2^{-\frac{n}{2p}}\|\eta_{\alpha} u\|_{L^p (X,V_g)} \leq 
\| (\eta_{\alpha} u) \circ \phi^{-1}_{\alpha}\|_{L^p (\mathbb{R}^n,\mathcal{L}^n)} \leq 2^{\frac{n}{2p}}\|\eta_{\alpha} u\|_{L^p (X,V_g)} .  
\end{eqnarray}
On the other hand, we estimate the gradient locally by
\begin{eqnarray*}
\int_{X}  \left| \nabla (\eta_{\alpha} u) \right|^{p}  d V_{g}
&=& \int_{\phi_{\alpha} (U_{\alpha})} \sqrt{ \det g^{\alpha}_{ij}}\ \left| \sum_{k,j=1}^{n} g_{\alpha}^{kj} D_{k} ( (\eta_{\alpha} u ) \circ \phi^{-1}_{\alpha}) D_{j} ( (\eta_{\alpha} u ) \circ \phi^{-1}_{\alpha}) \right|^{p}  d\mathcal{L}^n \nonumber \\
&\geq& 2^{- \frac{n+p}{2}} \int_{\phi_{\alpha} (U_{\alpha})} \left| \nabla ( (\eta_{\alpha} u ) \circ \phi^{-1}_{\alpha}) \right|^{ p }  d\mathcal{L}^n,
\end{eqnarray*}
therefore
\begin{equation*} 
 \| \nabla ( \eta_{\alpha} u) \circ \phi_{\alpha}^{-1} \|_{L^p(\mathbb{R}^n, \mathcal{L}^n)} \leq 2^{\frac{n+p}{2p}}  \| \nabla( \eta_{\alpha} u) \|_{L^p(X,V_g)} 
\end{equation*}
for each $\alpha$ in $A$.\\

Set $\displaystyle C_0 := \max_{\alpha \in A} \| \nabla \eta_{\alpha} \|_{\infty}+1 $. Then 
\begin{equation} \label{eq:1}
 \| \nabla ( \eta_{\alpha} u) \circ \phi_{\alpha}^{-1} \|_{L^p(\mathbb{R}^n, \mathcal{L}^n)} \leq 2^{\frac{n+p}{2p}}  \left( \|  \nabla u \|_{L^p(X, V_g)} + C_0 \|  u \|_{L^p(X,V_g)} \right) 
 \leq 2^{\frac{n+p}{2p}}C_0 \| u \|_{L^p_1(X,V_g)}   \ .
\end{equation}

Fix some $\epsilon >0$; then there exists a Haj\l{}asz gradient 
$h_{\alpha}$ for $(\eta_{\alpha} u) \circ \phi_{\alpha}^{-1} $ in $\phi_{\alpha}(U_{\alpha})$, so that 
\begin{equation}    \label{eq:2}
\| (\eta_{\alpha} u) \circ \phi_{\alpha}^{-1} \|_{L^p(\mathbb{R}^n,\mathcal{L}^n)} + \| h_{\alpha} \|_{L^p(\mathbb{R}^n,\mathcal{L}^n)} - \epsilon \leq \| (\eta_{\alpha} u)  \circ \phi_{\alpha}^{-1}\|_{M^p_1(\mathbb{R}^n, \mathcal{L}^n)}.
\end{equation}
Gather inequalities (\ref{ru}), (\ref{niermodularfunkc}), (\ref{eq:1}) and (\ref{eq:2}) to get
\begin{eqnarray*}
 \| (\eta_{\alpha} u) \circ \phi_{\alpha}^{-1} \|_{L^p(\mathbb{R}^n,\mathcal{L}^n)} + \| h_{\alpha} \|_{L^p(\mathbb{R}^n,\mathcal{L}^n)} - \epsilon &\leq&  C \| (\eta_{\alpha} u) \circ \phi_{\alpha}^{-1} \|_{L^p_1(\mathbb{R}^n,\mathcal{L}^n)} \\
& \leq& C \left( 2^{\frac{n+p}{2p}}C_0 + 2^{\frac{n}{2p}} \right) \| u \|_{L^p_1(X,V_g)}.
\end{eqnarray*}

Observe that for each $\alpha$ the function $\sqrt{2} h_{\alpha} \circ \phi_{\alpha} =: \tilde{h}_{\alpha}: U_{\alpha} \rightarrow \mathbb{R}$ is a Haj\l{}asz gradient for $ (\eta_{\alpha} u) |_{U_{\alpha}}$. 
Indeed, since $h_{\alpha}$ is a Haj\l{}asz gradient for $(\eta_{\alpha} u) \circ \phi_{\alpha} ^{-1}\  \mathbf{1}_{\phi_{\alpha} (U_{\alpha})}$, there exists  a subset $E_{\alpha} \subset \mathbb{R}^n$ such that $\mathcal{L}^n(E_{\alpha})=0$, and such that for every pair $x,y \in \phi_{\alpha} (U_{\alpha}) \thicksim E_{\alpha}$
\[
	|\eta_{\alpha} u(\phi_{\alpha}^{-1}(x)) - \eta_{\alpha} u(\phi_{\alpha}^{-1}(y))| \leq \left( h_{\alpha}( x ) + h_{\alpha}(y) \right) |x - y|. 
\]
Therefore, for each pair $x,y$ in  $U_{\alpha} \thicksim \phi_{\alpha} ^{-1}(E_{\alpha})$
\begin{eqnarray}   \label{gradient}
|\eta_{\alpha}(x) u(x) - \eta_{\alpha}(y) u(y)| &=& |\eta_{\alpha} u(\phi_{\alpha}^{-1}( \phi_{\alpha}(x) )) - \eta_{\alpha} u(\phi_{\alpha}^{-1}( \phi_{\alpha}(y))|  \nonumber \\
&\leq &\left( h_{\alpha}( \phi_{\alpha}(x) ) + h_{\alpha}(\phi_{\alpha}(y)) \right) |\phi_{\alpha}(x) - \phi_{\alpha} (y)|  \nonumber \\
& \leq &\left( \tilde{h}_{\alpha}(x) + \tilde{h}_{\alpha}(y) \right) d_g(x,y).
\end{eqnarray}

Our next goal is to prove that
 \[
 h:= \sum_{\alpha \in A} \tilde{h}_{\alpha} \mathbf{1}_{U_{\alpha}}
\]
is a local Haj\l{}asz gradient for $u$. \\ 

Fix $z\in X$ and define
\begin{itemize}
\item $I_z := \{ \alpha \in A: z \in U_{\alpha}\},$
\item $ J_z :=\{ \alpha \in A : z \in \partial U_{\alpha}\},$ and
\item $K_z :=\{ \alpha \in A: z \in X \thicksim \bar{U}_{\alpha}\}.$
\end{itemize}
Then $I_z, J_z, K_z$ are pairwise disjoint, with $I_z \cup J_z\cup K_z= A$ for each $z$ in $X$.\\

Define $R > 0$ such that
\begin{itemize}
\item For all  $\alpha$ in $I_z$, the ball $B_R(z) \subset U_{\alpha}$,
\item  For all  $\alpha$ in $J_z,   \eta_{\alpha} (B_R(z))=\{0\}$, and
\item For all $\alpha$ in $K_z,  B_R(z) \cap U_{\alpha} = \emptyset.$
\end{itemize}

Note that if $x,y \in B_R(z)$ and $\eta_{\alpha}(x) \neq 0$, then $y \in U_{\alpha}$; indeed, $\alpha$ can not belong to $K_z \cup J_z$, therefore $\alpha \in I_z$, and then $y \in B_R(z) \subset U_{\alpha}$. Hence for $x, y \in B_R(z)\thicksim
 \bigcup_{\alpha \in A} \phi^{-1}_{\alpha}(E_{\alpha})$, taking (\ref{gradient}) into account, the inequality
\[
\left| u(x) - u(y) \right| \leq \sum_{\alpha \in A} \left| \eta_{\alpha}(x) u (x) - \eta_{\alpha}(y) u (y) \right| \leq \sum_{\alpha \in A} \left(\tilde{h}_{\alpha}(x) + \tilde{h}_{\alpha}(y) \right) d_g(x,y),
\]
follows, and this proves that $h$ is a local Haj\l{}asz gradient for $u$. \\

Recalling  (\ref{local-norm}), collect previous  estimates to infer
\begin{eqnarray*}
\| u \|_{m^p_1(X, V_g)}  & \leq& \sum_{\alpha \in A} \| \eta_{\alpha}  u \|_{L^p(X,V_g)} + \| h \|_{L^p(X,V_g)} \\
&\leq&    2^{\frac{n}{2p}} \sum_{\alpha \in A}  \| (\eta_{\alpha} u) \circ \phi_{\alpha}^{-1} \|_{L^p(\mathbb{R}^n,\mathcal{L}^n)} + 2^{\frac{n}{2p}} \sum_{\alpha \in A} \|h_{\alpha}\|_{L^p(\mathbb{R}^n,\mathcal{L}^n)} \\
& \leq& 2^{\frac{n}{2p}}|A| C \left( 2^{\frac{n+p}{2p}}C_0 + 2^{\frac{n}{2p}} \right) \| u \|_{L^p_1(X,V_g)} +2^{\frac{n}{2p}}|A|\epsilon ,
\end{eqnarray*}
where $|A|$ denotes the cardinality of the set $A$. Hence if $ \epsilon \rightarrow 0 $ 
\begin{equation}   \label{ineq-m-M}
\| u \|_{m^p_1(X,V_g)} \leq C_1 \| u \|_{L^p_1(X,V_g)},
\end{equation}
where $C_1 := 2^{\frac{n}{2p}} |A| C \left( 2^{\frac{n+p}{2p}}C_0 + 2^{\frac{n}{2q}} \right)$.\\

{\bf Step 2.} Choose $u$ in $L^p_1(X, V_g)$. By the density of $C^{\infty}(X)$ in $L^p_1(X, V_g)$ there exists a sequence of smooth 
functions $\{u_n\}$ converging to $u$ in $L^p_1(X,V_g)$. Therefore using (\ref{ineq-m-M}) for every $\epsilon >0$ there exists some $N$ such that for $m,n \geq N$ 
\[
\| u_m - u_n \|_{m^{1,p}(X, V_g)} \leq C_1\| u_m - u_n \|_{L^p_1(X,V_g)} \leq \epsilon.
\]
On the other hand, by the completeness of  $m^p_1(X,V_g)$ the sequence $\{u_n\}$ converges to some $v$  in $m^p_1(X,V_g)$.
By the definitions of $L^p_1(X,V_g)$ and $m^p_1(X,V_g)$ the sequence $\{u_n\}$ converges to both $u$ and $v$ in $L^p(X,V_g)$: Thus $u = v$, and using (\ref{ineq-m-M})
\[
 \| u \|_{m^p_1(X,V_g)} \leq C_1 \| u \|_{L^p_1(X,V_g)}, 
\]
therefore $L^p_1(X,V_g)  \hookrightarrow  m^p_1(X,V_g)$. \\

Finally, by Corollary 3.5 from \cite{JSYY} the spaces $m^p_1(X,V_g)$ and $M^p_1(X,V_g)$ are equivalent, hence
\[
L^p_1(X,V_g) \hookrightarrow M^p_1(X,V_g), 
\]
as required.
\end{proof}


\subsection{Theorem \ref{main2} for flows} \label{final}

We use Theorem \ref{riem-mm} to apply Theorem \ref{main2} in closed Riemannian manifolds when the  $H$-orbits  have dimension one. In this setup we see that Theorem \ref{main2} can be applied if and only if the Euler characteristic of the manifold is equal to zero; this condition is restrictive only in even  dimensional manifolds.  \\

Consider a closed orientable Riemannian $n$-manifold $(X,g)$ whose Euler characteristic $\chi(X)$ is equal to $0$. A result attributed to H. Hopf, see \cite{milnor}, ensures that there exists a non-vanishing\footnote{Non-vanishing, or vanishing no-where.} vector field  
$\tau$ on $X$, or equivalently a non-vanishing $(n-1)$-form
$\omega_{\tau}$, related with $\tau$ through the bijection 
$TX \longleftrightarrow \wedge^{n-1} TX$  given by 
\[  \tau \leftrightarrow \omega_{\tau} = \tau \lrcorner\ \Omega_g ,
\]
where $\Omega_g$ is the volume $n$-form induced from $g$ giving the orientation of $X$.
The form $\omega_{\tau}$ is closed if and only if the vector field $\tau$ is free of divergence; indeed:
\[
( \text{div} \cdot \tau )\  \Omega_g = L_{\tau} \Omega_g = d (\tau \lrcorner\ \Omega_g),  
\]
where $L_{\tau}$ is the Lie derivative along $\tau$.
Denote by $H= \{ h_t : t \in \mathbb{R}\}$  the subgroup
of  $\text{Diff}(X)$ associated to the flow of $\tau$: if 
$\omega_{\tau}$ is a non-vanishing closed $(n-1)$-form on $X$, then  $H$ is a subgroup of $\text{Diff}_{V_g}(X)$, and the orbit of every point in $X$ under $H$ is uncountable. \\

In this spirit, D. Asimov proved in \cite{Asimov} that if $n$ is at least $4$,  and if the first Betti number of $X$ is different from zero, then every non-vanishing vector field is homotopic through a family of non-vanishing vector fields to a non-vanishing vector field that preserves $\Omega_g$, see also \cite{sullivan}. Shortly afterwards, M. Gromov using Convex Integration\footnote{See \cite{eliashberg} for a detailed exposition.} proved that if  $n$ is at least $3$, then every
non-vanishing $(n-1)$-form can be homotoped through non-vanishing forms to a non-vanishing exact form, with no restrictions on the first Betti number of $X$. Note that when $n=2$ the only possible manifold is the $2$-torus, and then the required vector fields are constant slope fields \cite{Asimov}. \\

With those facts, Theorem \ref{riem-mm} and Corollary \ref{corollary} provide  simple and concrete applications:

\begin{corol} Suppose $(X,g)$ is an orientable closed Riemannian manifold
with $\chi(X)=0$. If $\tau$ is a non-vanishing solenoidal vector field,
the problem
\[
\text{Min}\ \{\ \int_X\ \Big( | \nabla u |_g^p + |u|^p \Big)\ d V_g\ :\     
 u \in L^p_{1,H}(X,V_g)\ \text{and}\ \int_{X}\ |u|^{p^{\ast}}\ dV_g = 1\ \}
\]
has a solution whenever $1 < p < n$, where $H$ is the group associated to the flow of $\tau$.
\end{corol}


\end{document}